\documentclass[12pt, reqno]{amsart}
\usepackage[demo]{graphicx}
\usepackage{caption}
\usepackage{subcaption}
\usepackage{amsmath, amsthm, amscd, amsfonts, latexsym, amssymb, graphicx, color}
\usepackage[bookmarksnumbered, colorlinks, plainpages]{hyperref}
\usepackage{tikz}
\usepackage{tkz-berge}

\usetikzlibrary {positioning}

\makeatletter \oddsidemargin.9375in \evensidemargin \oddsidemargin
\newtheorem{thm}{Theorem}[section]
\newtheorem{lem}[thm]{Lemma}
\newtheorem{prop}[thm]{Proposition}

\theoremstyle{definition}
\newtheorem{defn}[thm]{Definition}

\theoremstyle{remark}
\newtheorem{rem}[thm]{Remark}
\numberwithin{equation}{section}

\begin{document}
\setcounter{page}{1}

\noindent \textbf{{\footnotesize  }}\\[1.00in]
\title[The large sum graph]{The large sum graph related to comultiplication  modules}
\author[H.\ Ansari-Toroghy]{H.\ Ansari-Toroghy$^1$}

\address{$^1$Department of
Mathematics, \newline \indent University
of Guilan, P.O.41335-19141, \newline \indent Rasht, Iran.}
\email{ansari@guilan.ac.ir}

\author[F.\ Mahboobi-Abkenar]{F.\ Mahboobi-Abkenar$^2$}

\address{$^2$Department of
Mathematics, \newline \indent University
of Guilan, P.O.41335-19141, \newline \indent Rasht, Iran.}
\email{mahboobi@phd.guilan.ac.ir}


\begin{abstract}
Let $R$ be a commutative ring and $M$ be an $R$-module. We define the large sum graph, denoted by $\acute{G}(M)$, as a graph with the vertex set of non-large submodules of $M$ and two distinct vertices are adjacent if and only if $N+K$ is a non-large submodule of $M$. In this article, we investigate the connection between the graph-theoretic properties of $\acute{G}(M)$ and some algebraic properties of $M$ when $M$ is a comultiplication $R$-module.
\end{abstract}

\maketitle
\section{Introduction}
Throughout this paper, $R$ will denote a commutative ring with identity and $\Bbb Z$ will denote the ring of integers.

 Let $M$ be an $R$-module. We denote the set of all minimal submodules of $M$ by $Min(M)$ and the sum of all minimal submodules of $M$ by $Soc(M)$. A submodule $N$ of $M$ is called \emph{large} in $M$  and denoted by $N \unlhd M$) in case for every submodule $L$ of $M$, $N\cap L \neq 0$. A module $M$ is called a \emph{uniform} module if the intersection of any two non-zero submodules of $M$ is non-zero.

 A \emph{graph} $G$ is defined as the pair $(V(G), E(G))$, where $V(G)$ is the set of vertices of $G$ and $E(G)$ is the set of edges of $G$. For two distinct vertices $a$ and $b$ denoted by $a-b$ means that $a$ and $b$ are adjacent. The \emph{degree} of a vertex $a$ of graph $G$ which denoted by $deg(a)$ is the number of edges incident on $a$. A regular graph is \emph{$r$-regular} (or regular of degree r) if the degree of each vertex is $r$. If $|V(G)|\geqslant 2$, a \emph{path} from $a$ to $b$  is a series of adjacent vertices $a-v_1-v_2-...-v_n-b$. In a graph $G$, the distance between two distinct vertices $a$ and $b$, denoted by $d(a, b)$ is the length of the shortest path connecting $a$ and $b$. If there is not a path between $a$ and $b$, $d(a, b) = \infty$. The \emph{diameter} of a graph $G$ is $diam(G) = sup\,\{d(a,b)\,|\:a, b\in V(G)\}$. A graph $G$ is called \emph{connected} if for any vertices $a$ and $b$ of $G$ there is a path between $a$ and $b$. If not, $G$ is \emph{disconnected}. The \emph{girth} of $G$, denoted by $g(G)$, is the length of the shortest cycle in $G$. If $G$ has no cycle, we define the girth of $G$ to be infinite. An \emph{$r$-partite} graph is one whose vertex set can be partitioned into $r$ subsets such that no edge has both ends in any one subset. A \emph{complete} $r$-partite graph is one each vertex is jointed to every vertex that is not in the same subset. The \emph{complete bipartite} (i.e, 2-partite) graph with part sizes $m$ and $n$ is denoted by $K_{m,n}$. A \emph{star} graph is a completed bipartite graph with part sizes $m =1$ or $n = 1$. A \emph{clique} of a graph is its maximal complete subgraph and the number of vertices in the largest clique of a graph $G$, denoted by $\omega(G)$, is called the \emph {clique number} of $G$. For a graph $G = (V, E)$, a set $S \subseteq V$ is an \emph{independent} if no two vertices in $S$ are adjacent. The \emph{independence number} $\alpha(G)$ is the maximum size of an independent set in $G$.  The (open) \emph{neighbourhood} $N(a)$ of a vertex $a \in V$ is the set of vertices which are adjacent to $a$. For each $S \subseteq V$, $N(S) = \bigcup_{a \in S}N(a)$ and $N[S] = N(S) \bigcup S$. A set of vertices $S$ in $G$ is a \emph{dominating set}, if $N[S] = V$. The \emph{dominating number}, $\gamma(G)$, of $G$ is the minimum cardinality of a dominating set of $G$ (\cite{7}).  Note that a graph whose vertices-set is empty is a \emph{null graph} and a graph whose edge-set is empty is an \emph{empty} graph.

 A module $M$ is said to be a \emph{comultiplication} $R$-module if for every submodule $N$ of $M$ there exists an ideal $I$ of $R$ such that $N = Ann_M(I)$. Also an $R$-module $M$ is comultiplication module if and only if for each submodule $N$ of $M$, we have $N = (0 :_M Ann_R(N))$ (\cite{1}).

In this article, we introduce and study \emph{the sum  large graph} $\acute{G}(M)$ of $M$, where $M$ is a comultiplication module. In section 2, we give the definition of $\acute{G}(M)$ and consider some basic results on the structure of this graph. In Theorems \ref{2.5} and \ref{2.6}, we provide some useful characterization about $\acute{G}(M)$. In Theorem \ref{2.7}, it is shown that if $\acute{G}(M)$ is connected, then $diam(\acute{G}(M)) \leqslant 2$. Also we prove that if $\acute{G}(M)$ contains a cycle, then $g(\acute{G}(M))= 3$ (Theorem \ref{2.8}).
Moreover, it is proved that if $\acute{G}(M)$ is a connected graph, then $\acute{G}(M)$ has no cut vertex (Theorem \ref{2.9}).
Finally, in section 3, we investigate the clique number, dominating number, and independence number of this graph.

\vspace{0.5 cm}
\section{Basic properties of $\acute{G}(M)$}

\begin{defn}\label{2.1}
Let $M$ be an $R$-module. We define the \emph{large sum graph $\acute{G}(M)$  of $M$ }  with all non-large non-zero submodules of $M$ as vertices and two distinct vertices $N, K$ are adjacent if and only if $N + K$ is a non-large submodule of $M$.
\end{defn}
\vspace{0.2 cm}
A non-zero submodule $S$ of $M$ is said to be \emph{second} if for each $a \in R$, the endomorphism of $M$ given by multiplication by $a$ is either surjective or zero. This implies that $Ann_R(N)$ is a prime ideal of $R$ (\cite{4}).\\
\\
The next lemma plays a key role in the sequel.
\vspace{0.5 cm}
\begin{lem}\label{2.2}
Let $M$ be a non-zero comultiplication $R$-module.
\begin{itemize}
  \item [(a)] Every non-zero submodule of $M$ contains a minimal submodule of $M$. In particular, $Min(M) \neq \varnothing$.
    \item [(b)] Let $N$ be a submodule of $M$. Then $N$ is a large submodule of $M$ if and only if $Soc(M)\subseteq N$.
    \item [(c)] Let $N, K$ be submodules of $M$ and  let $S$ be a second submodule of $M$ with $S\supseteq N + K$.  Then $S\supseteq N$ or $S\supseteq
        K$.
\end{itemize}
\end{lem}
\begin{proof}(a) See \cite[Theorem 3.2]{2}.

(b) Let $N$ be a large submodule of $M$. Assume to the contrary that $Soc(M) \nsubseteq N$. Then for each $S_j \in Min(M)$, we have  $S_j \nsubseteq N$. Since $N$ is a large submodule of $M$, $N \cap S_j \neq 0$. Since $S_j$ is a minimal submodule of $M$ and $N\cap S_j \subseteq S_j$, we have $S_j\cap N = S_j$, which is a contradiction. Conversely, suppose to the contrary that $N$ is not a large submodule of $M$. Then there exists a submodule $K$ of $M$ such that $N \cap K = 0$. By part (a), there exists a minimal submodule $L$ of $M$ such that $L \subseteq K$. So we have $L\cap N \subseteq K\cap N = 0$, which implies that $L \nsubseteq N$, a contradiction. Thus $N$ is a large submodule of $M$.

(c) See \cite[Theorem 2.6] {8}.
\end{proof}
\vspace{0.25 cm}

In the rest of this paper, we assume that $M$ is a non-zero comultiplication $R$-module. We recall that $Min(M) \neq \varnothing$ by Lemma
\ref{2.2} part (a).
\vspace{0.25 cm}
\begin{lem}\label{2.3}
Let $M$ be an $R$-module with $Min(M) = \{S_i\}_{i \in I}$, where $|I| > 1$, and let $\Lambda$ be a non-empty proper finite subset of $I$. Then  $\sum_{\lambda \in \Lambda} S_\lambda$ is non-large submodule of $M$.
 \end{lem}
\begin{proof}
 Let $\sum_{\lambda \in \Lambda}S_{\lambda}$ be a large submodule of $M$ and let $j \in I \setminus \Lambda$. Then  by Lemma  \ref{2.2} (b), $S_j\subseteq \sum_{\lambda \in \Lambda} S_{\lambda}$. Since $S_j$ is a second submodule of $M$, by Lemma \ref{2.2} (c), $S_j\subseteq S_{\lambda} $ for some $\lambda \in \Lambda$, a contradiction.
\end{proof}
\vspace{0.25 cm}
We recall that an $R$-module $M$ is said to be \emph{finitely cogenerated} if for every set $\{M_\lambda\}_{\lambda \in \Lambda}$ of submodules of $M$, $\cap_{\lambda \in \Lambda}M_\lambda = 0$ implies $\cap^n_{i=1}M_{\lambda_i} = 0$ for some positive integer $n$ (\cite{3}). Moreover, an $R$-module $M$ is said to be \emph{cocyclic} if Soc(M) is a large and a simple submodule of $M$ (\cite{4}).
\vspace{0.5 cm}
\begin{prop}\label{2.4}
Let $M$ be an $R$-module. Then $\acute{G}(M)$ is a null graph if and only if $M$ is a cocyclic module.
\end{prop}
\begin{proof}
This is straightforward.
\end{proof}
\vspace{0.5 cm}
Note that all definitions graph theory are for non-null graphs \cite{6}. So in the rest of this paper we assume that $\acute{G}(M)$ is a non-null graph.
\vspace{0.5 cm}

\begin{lem}\label{2.51}
Let $M$ be an $R$-module. Then $M$ is uniform if and only if $M$ is a cocyclic $R$-module.
\end{lem}
\begin{proof}
This is obvious.
\end{proof}
\begin{thm}\label{2.5}
Let $M$ be an $R$-module. Then $\acute{G}(M)$ is an empty graph if and only if $Min(M) = \{S_1, S_2\}$ such that $\frac{M}{S_1}$ and $\frac{M}{S_2}$ are finitely cogenerated uniform $R$-modules.
\end{thm}
\begin{proof}
Let $\acute{G}(M)$ be an empty graph. If $|Min(M)| > 2$, then by Lemma \ref{2.3}, $S_1, S_2 \in Min(M)$ are adjacent, a contradiction. Thus $Min(M) =\{S_1, S_2\}$. Now we claim that $\frac{S_1+S_2}{S_1}$ is a minimal submodule of $\frac{M}{S_1}$ because $\frac{S_2}{S_1 \cap S_2} \simeq \frac{S_1 + S_2}{S_1}$, where $S_1 \cap S_2 = 0$. We show that $\frac{S_1 + S_2}{S_1}$ is the only minimal submodule of $\frac{M}{S_1}$. Suppose that $\frac{K}{S_1}$ is a minimal submodule of $\frac{M}{S_1}$. If $\frac{K}{S_1}$  is a large submodule of $\frac{M}{S_1}$, then $\frac{S_1 + S_2}{S_1} = \frac{Soc(M)}{S_1} \subseteq \frac{K}{S_1}$ which implies that $\frac{S_1 + S_2}{S_1} = \frac{K}{S_1}$, a contradiction. Now assume that $\frac{K}{S_1}$  is a non-large submodule of $\frac{M}{S_1}$. Then $K$ is a non-large submodule of $M$. So we have $K + S_1 = K \ntrianglelefteq M$ (i.e., $K$ is a non-large submodule of $M$) which follows that $K$ and $S_1$ are adjacent, a contradiction. Thus $\frac{S_1 + S_2}{S_1}$ is the only minimal submodule of $\frac{M}{S_1}$, and therefore $\frac{M}{S_1}$ is a uniform module. Since $Soc(\frac{M}{S_1}) = \frac{S_1 + S_2}{S_1}$ is a simple module, $\frac{M}{S_1}$ is a finitely cogenerated module by \cite [Proposition 10.7]{3}. Conversely, let $Min(M) = \{S_1, S_2\}$. Then clearly, $S_1$ and $S_2$ are not adjacent. We claim that there is no vertex $N \neq S_1, S_2$.  Assume to the contrary that $N$ is  a vertex of $\acute{G}(M)$. By Lemma \ref{2.2} (a), $S_1 \subseteq N$ or $S_2 \subseteq N$. Without loss of generality we can assume that $S_1 \subseteq N$. One can see that $\frac{S_1 + S_2}{S_1}$ is a minimal submodule of $\frac{M}{S_1}$. Since $\frac{M}{S_1}$ is a cocyclic module by lemma \ref{2.51}, for any submodule $\frac{N}{S_1}$, we have $\frac{S_1 + S_2}{S_1} = Soc(\frac{M}{S_1}) \subseteq \frac{N}{S_1}$. Thus each submodule $N$ of $M$ is a large submodule by Lemma \ref{2.2} (b), a contradiction. Hence $\acute{G}(M)$ is an empty graph.
\end{proof}

\begin{thm}\label{2.6}
Let $M$ be an $R$-module. The following statements are equivalent.
\begin{itemize}
\item[(i)] $\acute {G}(M)$ is not connected.
\item[(ii)] $|Min(M)| = 2$
\item[(iii)] $\acute {G}(M) = \acute{G_1} \cup \acute{G_2}$, where $\acute{G_1}$ and $\acute{G_2}$ are complete and disjoint subgraphs.
\end{itemize}
\end{thm}

\begin{proof}
$(i)\Rightarrow (ii)$ Assume to the contrary that $|Min(M)| > 2$. Since $\acute{G}(M)$ is not connected, we can consider two components $\acute{G_1}, \acute{G_2}$ and $N, K$ two submodules of $M$ such that $N \in \acute{G_1}$ and $K \in \acute{G_2}$. Choose $S_1, S_2 \in Min(M)$ such that $S_1 \subseteq N$ and $S_2 \subseteq K$. If $S_1 = S_2$, then $N-S_1-K$ is a path, a contradiction. So we can assume that $S_1 \neq S_2$. Since $Min(M) > 2$, $S_1 + S_2$ is a non-large submodule  of $M$ by Lemma \ref{2.3}. Thus $N-S_1-S_2-K$ is a path between $\acute{G_1}$ and $\acute{G_2}$, a contradiction. Therefore, $|Min(M)| = 2$.\\
$(ii)\Rightarrow (iii)$ Let $Min(M) = \{S_1, S_2\}$. Set $\acute{G_j} := \{N \leq M\, \mid\, N \subseteq S_j\, and \,N \not\trianglelefteq M\}$, where $j = 1, 2$. Assume that $N, K \in \acute{G_1}$. We claim that $N$ and $K$ are adjacent. Otherwise, If $N+K\trianglelefteq M $, then $S_1 + S_2 = Soc(M) \subseteq N+K$ by Lemma \ref{2.2} (b). Now by using Lemma \ref{2.2} (c), $S_1 + S_2 \subseteq N$ or $S_1 + S_2 \subseteq K$ . So $N$ or $K$ are large submodules of $M$, a contradiction. By using similar arguments for $\acute{G_2}$, we can conclude that $\acute{G_1}, \acute{G_2}$ are complete subgraphs of $\acute{G}(M)$. We claim that these two subgraphs are disjoint. Assume to the contrary that $N_1 \in \acute{G_1}$ and $N_2 \in \acute{G_2}$ are adjacent. Then  $Soc(M) = S_1 + S_2\subseteq N_1 + N_2$ which implies that $N_1 + N_2$ is a large submodule of $M$ by Lemma \ref{2.2} (b), a contradiction.\\
$(iii) \Rightarrow (i)$ This is obvious.
\end{proof}
\vspace{0.25 cm}
\begin{rem}\label{1}
 The condition that ``$M$ is a comultiplication module'' can not be omitted in Theorem \ref{2.6}. For example, let $M = \Bbb Z_2 \oplus \Bbb Z_4$ be a $\Bbb Z$-module and let $N_1:=\:  (0, 1)\Bbb Z,\, N_2:=\:  (0, 2)\Bbb Z,\, N_3:=\:  (1, 0)\Bbb Z,\, N_4:=\:  (1,1)\Bbb Z,\,\, and\, N_5:=\: (1, 2)\Bbb Z$. Then $V(\acute{G}(M)) = \{N_1, N_2, N_3, N_4, N_5\}$ and $Min(M) = \{N_2, N_3, N_5\}$. Thus $|Min(M)|> 2$ but  $\acute{G}(M)$ is not a connected graph.
\end{rem}

\begin{thm}\label{2.7}
Let $\acute{G}(M)$ be a connected graph. Then $diam(\acute{G}(M)) \leqslant 2$.
\end{thm}
 \begin{proof}
 Let $N$ and $K$ be two vertices of $\acute{G}(M)$ such that they are not adjacent. By Lemma \ref{2.2} (a), there exist two minimal submodules $S_1, S_2$ of $M$ such that $S_1 \subseteq N$ and $S_2 \subseteq K$. If $N + S_2 \ntrianglelefteq M$, then $N-S_2-K$ is a path. So $d(N, K ) = 2$. Similarly, if $K + S_1 \ntrianglelefteq M$, then $d(N, K) = 2$. Now assume that $N + S_2 \trianglelefteq M$ and $K + S_1 \trianglelefteq M$. By Theorem \ref{2.6}, $|Min(M)| \geq 3$ because $\acute{G}(M)$ is a connected graph. Let $S_3$ be a minimal submodule of $M$ such that $S_3 \neq S_1, S_2$. Thus by Lemma \ref{2.2} (b),  we have $S_3 \subseteq Soc(M) \subseteq N + S_3$ and $S_3 \subseteq Soc(M) \subseteq K + S_1$. Now Lemma \ref{2.2} (c) shows that $S_3 \subseteq N$ and $S_3 \subseteq K$. Hence we have $N-S_3-K$. Therefore, $d(N, K) = 2$.
 \end{proof}

 \begin{thm}\label{2.8}
 Let $M$ be an $R$-module and $\acute{G}(M)$ contains a cycle. Then $g(\acute{G}(M)) = 3$.
 \end{thm}
 \begin{proof}
 If $|Min(M)| = 2$, then $\acute{G}(M) = \acute{G_1} \cup \acute{G_2}$, where $\acute{G_1}$ and $\acute{G_2}$ are complete disjoint subgraphs by Theorem \ref{2.6}. Since $\acute{G}(M)$  contains a cycle and $\acute{G_1}, \acute{G_2}$ are disjoint complete subgraphs, $g(\acute{G}(M)) = 3$. Now assume that $|Min(M)| \geq 3$ and choose $S_1, S_2,\: and\: S_3 \in Min(M)$. By Lemma \ref{2.3}, $S_1-S_2-S_3-S_1$ is a cycle. Hence $g(\acute{G}(M)) = 3$.
 \end{proof}

 A vertex $a$ in a connected graph $G$ is a \emph{cut vertex} if $G-\{a\}$ is disconnected.

 \begin{thm}\label{2.9}
 Let $M$ be an $R$-module. If $\acute{G}(M)$ is a connected graph, then $\acute{G}(M)$ has no cut vertex.
 \end{thm}
 \begin{proof}
 Assume on the contrary that there exists a vertex $N \in V(\acute{G}(M))$ such that $\acute{G}(M) \setminus N$ is not connected. Thus there exist at leat two vertices $K, L$ such that $N$ lies in every path between them. By Theorem \ref{2.7}, the shortest path between $K$ and $L$ is length of two. So we have $K-N-L$. Firstly, we claim that $N$ is a minimal submodule of $M$. Otherwise, there exists a minimal submodule $S$ of $M$ such that $S \subset N$ by Lemma \ref{2.2} (a). Since $S + K \subseteq N + K$ and $N + K \not\trianglelefteq M$, we have $S + K \ntrianglelefteq M$. By similar arguments, $S + L$ is a non-large submodule of $M$. Hence $K-S-L$ is a path in $\acute{G}(M) \setminus N$, a contradiction. Thus $N$ is a minimal submodule of $M$. Now we claim that there is a minimal submodule $S_i \neq N$ such that $S_i \nsubseteq K$. Suppose on the contrary that $S_i \subseteq K$ for each $S_i \in Min(M)$. So we have $Soc(M) \subseteq K + N$. This implies that $K + N$ is a large submodule of $M$ by Lemma \ref{2.2} (b), a contradiction. Similarly, there exits a minimal submodule $S_j \neq N$ of $M$ such that $S_j \nsubseteq L$. Note that for each $S_t \in Min(M)$, we have $S_t \subseteq K + L$ because $K + L$ is a large submodule of $M$. So $S_t \subseteq K$ or $S_t \subseteq L$ by Lemma \ref{2.2} (c). Now let $N \neq S_i, S_j \in Min(M)$ such that $S_i \nsubseteq K$ and $S_j\nsubseteq L$ (Note that $S_i \neq S_j$). Hence we have $S_i \subseteq L$ and $S_j \subseteq K$. This implies that $K-S_i-S_j-L$ is a path in $\acute{G}(M) \setminus L$, a contradiction.
 \end{proof}

 \begin{thm}\label{2.10}
 Let $M$ be an $R$-module. Then $\acute{G}(M)$ can not be a complete $n$-partite graph.
 \end{thm}
 \begin{proof}
 Suppose on the contrary that $\acute{G}(M)$ is a complete $n$-partite graph with parts $U_1, U_2, ..., U_n$. By Lemma \ref{2.3}, for every $S_i, S_j \in Min(M)$, $S_i$ and $S_j$ are adjacent. Hence each $U_i$ contains at most one minimal submodule. By Pigeon hole principal, $|Min(M)| \leqslant n$. Now we claim that $|Min(M)| = t$ where $t < n$. Let $S_i \in V_i$, for each $i$ ($1\leq i \leq t$). Then $V_{t+1}$ contains no minimal submodule of $M$. By Lemma \ref {2.3}, $\Sigma_{j \neq i}S_j$ is a non-large submodule of $M$. Clearly, $\Sigma_{j \neq i}S_j$ and $S_i$ are not adjacent  because $Soc(M) = \Sigma_{j \neq i}S_j + S_i$. Hence $\Sigma_{j \neq i}S_j \in V_i$. Let $N$ be a vertex in $V_{t + 1}$. Then by Lemma \ref{2.2} (a), there exists $S_k \in Min(M)$ such that $S_k \subseteq N$. So $N$ and $S_k$ are adjacent, where $S_k \in V_k$. Since $\acute{G}(M)$ is a complete $n$-partite graph,  $N$ adjacent to all vertices in $V_k$. So $N$ and $\Sigma_{j \neq k} S_j$ are adjacent. However, $Soc(M) = S_k + \Sigma_{j \neq k}S_j$ which implies that $N + \Sigma_{j \neq k}S_j \trianglelefteq M$ by Lemma \ref{2.3}, a contradiction. Hence $|Min(M)| = t$. Now set $K := \Sigma_{i = 3}S_i$. By Lemma \ref{2.3}, $K$ is a non-large submodule of $M$. Since $K + S_1 = \Sigma_{i \neq 2}S_i \ntrianglelefteq M$, $K$ and $S_1$ are adjacent. Similarly, $K$ is adjacent to $S_2$. Thus $K \not\in V_1, V_2$. Furthermore, $K + S_i = K \ntrianglelefteq M$ for each $i$ ($1\leq i\leq n$). Hence $K$ is adjacent to all  minimal submodules $S_i$ of $M$. So for each $i$ ($1\leq i\leq n$), $K \not \in V_i$, a contradiction.
 \end{proof}
 \begin{prop}\label{2.11}
 Let $M$ be an $R$-module with $|Min(M)| < \infty$. Then we have the following.
 \begin{itemize}
 \item [(i)] There is no vertex in $\acute{G}(M)$ which is adjacent to every other vertex.
 \item [(ii)] $\acute{G}(M)$ can not be a complete graph.
 \end{itemize}
  \end{prop}
  \begin{proof}(i)    Assume on the contrary that there exists a submodule $N \in V(\acute{G}(M))$ such that $N$ is adjacent to all vertices of $\acute{G}(M)$. By Lemma \ref{2.2} (a), there is a minimal submodule $S_i \in Min(M)$ such that $S_i \subseteq N$. Now set $K := \Sigma_{j \neq i} S_j$. Clearly, $K \ntrianglelefteq M$ by Lemma \ref{2.3}. Since $N$ is adjacent to all other vertices of $\acute{G}(M)$, $N + K$ is a non-large submodule of $M$. However, $Soc(M) = \Sigma_{j \neq i}S_j+ S_i \subseteq N + K$ which shows that $N + K \trianglelefteq M$ by Lemma \ref{2.2} (b), a contradiction.\begin{itemize}
    \item [(ii)] This follows from (i).
  \end{itemize}
  \end{proof}
\vspace{0.25 cm}
A vertex of a graph $G$ is said to be \emph{pendent} if its neighbourhood contains exactly one vertex.
\vspace{0.25 cm}

  \begin{thm}\label{2.12}
  Let $M$ be an $R$-module. Then we have the following.
\begin{itemize}
 \item[(i)]
 $\acute{G}(M)$ contains a pendent vertex if and only if $|Min(M)| = 2$ and $\acute{G}(M) = \acute{G_1} \cup \acute{G_2}$, where $\acute{G_1}, \acute{G_2}$ are two disjoint complete subgraphs and $|V(\acute{G_i})| = 2$ for some $i = 1, 2$.
 \item[(ii)]
 $\acute{G}(M)$ is not a star graph.
\end{itemize}
  \end{thm}
  \begin{proof}

 (i)  Let $N$ be a pendent vertex of $\acute{G}(M)$. Assume on the contrary that $|Min(M)| \geq 3$. Clearly, for each $S_i \in Min(M)$, $S_i$ is adjacent to every other minimal submodules of $M$. So $deg(S_i) \geq 2$. Thus $N$ is not a minimal submodule of $M$. By Lemma \ref{2.2} (a), there exists a minimal submodule of $S_1$ of $M$ such that $S_1 \subseteq N$. Note that the only vertex which is adjacent to $N$ is $S_1$ because $deg(N) = 1$. Hence there is no minimal submodule $S_i \neq S_1$ such that $S_i \subseteq N$. Moreover, $N + S_2$ is a large submodule of $M$. So by Lemma \ref{2.2} (b), $S_j \subseteq Soc(M) \subseteq N + S_2$, for each $S_j \neq S_1, S_2$. This implies that $S_j \subseteq N$ by Lemma \ref{2.2} (c),  a contradiction. Hence $|Min(M)| = 2$. By Theorem \ref{2.6}, $\acute{G}(M) = \acute{G_1} \cup \acute{G_2}$, where $\acute{G_1}$ and $\acute{G_2}$ are disjoint complete subgraphs. This is easy to see that $|V(\acute{G}_i)| = 2$. The converse is straightforward.\begin{itemize}
  \item[(ii)] Suppose that $\acute{G}(M)$ is a star graph. Then $\acute{G}(M)$ has a pendent vertex. So by part (i), we have $|Min(M)| = 2$. Thus $\acute{G}(M)$ is not a connected graph by Theorem \ref{2.6}, a contradiction.\end{itemize}
  \end{proof}
\begin{thm}
Let $M$ be an $R$-module.\begin{itemize}
\item[(i)] Let $N$, $K$ be two vertex of $\acute{G}(M)$ such that $N \subseteq K$. Then $deg(N) \leq deg(K)$.
\item[(ii)] Let $\acute{G}(M)$ be a $r$-regular graph. Then $|Min(M)| = 2$ and $|V(\acute {G}(M))| = 2r + 2$.\end{itemize}
\end{thm}
\begin{proof}(i)
Let $N, K \in V(\acute{G}(M))$ such that $N \subseteq K$. Let $L$ be a vertex of $\acute{G}(M)$ such that $L$ is adjacent to $N$. Thus $N + L$ is a non-large submodule of $M$ and so that $K + L$ is a non-large submodule of $M$. So $L$ is adjacent to $K$. Therefore, $deg(N) \leq deg(K)$.\begin{itemize}
\item[(ii)] Suppose on the contrary that $|Min(M)| \geq 3$. By using Lemma \ref{2.3} and our assumption, $Min(M)$ is a finite set. Next for $S_1, S_2 \in Min(M)$, we have $deg(S_1) \leq deg(S_1 + S_2)$ by part (a). We claim that $deg(S_1 + S_2) \neq deg(S_1)$. Otherwise, if $N := \Sigma_{j\neq 2} S_j$, then $N$ is adjacent to $S_1$. However, $N$ is not adjacent to $S_1 + S_2$, a contradiction. So $r < deg(S_1 + S_2)$, which is a contradiction. Thus $|Min (M)| \leq 2$. If $|Min(M)| = 1$, then $\acute{G}(M)$ is null graph, a contradiction. Thus $|Min(M)| = 2$ and so that by Theorem \ref{2.5} $\acute{G}(M) = \acute{G}_1 \cup \acute{G}_2$, where $\acute{G}_1, \acute{G}_2$ are disjoint complete subgraphs. Set $Min(M) = \{S_1, S_2\}$ and $S_i \in G_i$. Since $\acute{G}(M)$ is a $r$-regular graph, $|V(\acute{G_i})| = r +1$ for $i = 1, 2$. Hence we have $V(\acute{G}(M)) = 2r +2$.
    \end{itemize}
\end{proof}
\vspace{0.5 cm}
\section{clique number, dominating number, and independence number}
In this section, we obtain some results on the clique, dominating, and independence numbers of $\acute{G}(M)$.

\begin{prop}\label{3.1}
Let $M$ be an $R$-module. Then the following hold.
\begin{itemize}
  \item [(i)] Let $\acute{G}(M)$ be a non-empty graph. Then $\omega(\acute{G}(M)) \geq |Min(M)|$.
  \item [(ii)] Let $\acute{G}(M)$ be an empty graph. Then $\omega(\acute{G}(M)) = 1$ if and only if $Min(M) = \{S_1, S_2\}$, where $S_1$ and $S_2$ are finitely cogenerated uniform $R$-modules.
   \item [(iii)] If $\omega(\acute{G}(M)) < \infty$, then $\omega(\acute{G}(M)) \geq 2^{|Min(M)|-1}-1$.
\end{itemize}
\end{prop}
\begin{proof}
(i) If $|Min(M)| = 2$, then $\omega(\acute{G}(M)) \geq 2$ by Theorem \ref{2.6}. Now let $|Max(M)| \geq 3$. Then by Lemma \ref{2.3}, the subgraph of $\acute{G}(M)$ with the vertex set of $\{S_i\}_{S_i \in Min(M)}$ is a complete subgraph of $\acute{G}(M)$. So $\omega(\acute{G}(M)) \geq |Min(M)|$.
\begin{itemize}
  \item [(ii)] This follows directly from Theorem \ref{2.5}.

  \item [(iii)] Since $\omega(\acute{G}(M)) < \infty$, we have $|Min(M)| < \infty$ by part (i), (ii). Let $Min(M) = \{S_1,..., S_t\}$. For each $1\leq i\leq t$, consider
  \begin{center}
  $A_i = \{S_1,..., S_{i-1}, S_{i+1}, S_t\}$.
 \end{center}
 Now let $P(A_i)$ be the power set of $A_i$ and for each $X \in P(A_i)$, set $S_X = \bigcap_{S_j\in X}S_j$ for $1 \leq j\leq t$. The subgraph of $\acute{G}(M)$ with the vertex set $\{S_X\}_{X\in P(A_i)\setminus\{\emptyset\}}$ is a complete subgraph of $\acute{G}(M)$ by Lemma \ref{2.3}. It is clear that $|\{S_X\}_{X\in P(A_i)\setminus \{\emptyset\}}| =
      2^{|Min(M)|-1}-1$. Thus $\omega(\acute{G}(M)) \geq 2^{|Min(M)|-1}-1$.
\end{itemize}
\end{proof}
\vspace{0.3 cm}
\begin{rem}
Note that the condition  ``$M$ is a comultiplication module" is necessary in Proposition \ref{3.1}. For example, let $M = \Bbb Z_2\oplus\Bbb Z_4$ be as a $\Bbb Z$-module which is not a comultiplication module. Then $\omega(\acute {G}(M)) = 2$ but  $|Min(M)| = 3$.
\end{rem}
\vspace{0.2 cm}
\begin{thm}
Let $M$ be an $R$-module. Then $\gamma(\acute{G}(M))\leq 2$. In particular, if $|Min(M)| < \infty$, then $\gamma(\acute{G}(M)) = 2$.
\end{thm}
\begin{proof}
Clearly, $|Min(M)| \geq 2$ because $\acute{G}(M)$ is a non-null graph. Consider $S = \{S_1, S_2\}$, where $S_1, S_2 \in Min(M)$. Let $N$ be a vertex of $\acute{G}(M)$. We claim that $N$ is adjacent to $S_1$ or $S_2$. If $S_1 \subseteq N$ or $S_2 \subseteq N$, then the claim is true. Now assume that $S_1\nsubseteq N$ and $S_2 \nsubseteq N$. In this case, we also claim that $N$ is adjacent to $S_1$ or $S_2$. Without loss of generality, we can assume that $N$ is not adjacent to $S_1$. So $S_2 \subseteq Soc(M) \subseteq N$ by Lemma \ref{2.2} (b). This shows that $S_2 \subseteq N$, which is a contradiction. By similar arguments, we can show that $N$ is adjacent to $S_2$. Thus $\gamma(\acute{G}(M)) \leq 2$. The last assertion follows from Theorem \ref{2.11}.
\end{proof}

\begin{thm}
Let $M$ be an $R$-module and $|Min(M)| < \infty$. Then $\alpha(\acute{G}(M)) = |Min(M)|$.
\end{thm}
\begin{proof}
Let $Min(M) = \{S_1, ..., S_n\}$. It is easy to see that $\{\Sigma^n_{j = 1, j\neq i}S_j\}^n_{i=1}$ is an independent set. So $\alpha(\acute{G}(M))\geq n$. Now let $\alpha(\acute{G}(M)) = m$ and $S = \{N_1, ..., N_m\}$ be a maximal independent set. We claim that $m = n$. Assume on the contrary that $m > n$. Let $S_t \in Min(M)$. Then by Pigeon hole principal, there exist $N_i, N_j \in Min(M)$, where $1\leq i, j\leq n$, such that $S_t\nsubseteq N_i$ and $S_t \nsubseteq N_j$. Thus by using Lemma \ref{2.2} (c), we have $S_t \nsubseteq N_i+N_j$. Since $N_i, N_j \in S$ and $S$ is an independent set, we have $S_t \subseteq Soc(M)\subseteq N_i + N_j$. Then by Lemma \ref{2.2} (c), $S_t \subseteq N_i$ or $S_t \subseteq N_j$, which is a contradiction. Hence $\alpha(\acute{G}(M)) = n$.
\end{proof}
\vskip 0.4 true cm

\begin{center}{\textbf{Acknowledgments}}
\end{center}
We would like to thank Dr. F. Farshadifar for some helpful comments. \\ \\
\vskip 0.4 true cm

\bibliographystyle{amsplain}

\end{document}